\documentclass[numbook,envcountsame,smallcondensed,draft]{amsart}

\usepackage{amssymb,amsmath,amsfonts,cite}

\DeclareMathOperator{\con}{con}

\DeclareMathOperator{\End}{End}

\DeclareMathOperator{\FIC}{FIC}

\DeclareMathOperator{\occ}{occ}

\DeclareMathOperator{\Part}{Part}

\newtheorem*{theorem}{Theorem}

\newtheorem{lemma}{Lemma}

\theoremstyle{definition}

\newtheorem*{question}{Question}

\makeatletter

\renewcommand*\subjclass[2][2010]{\def\@subjclass{#2}\@ifundefined{subjclassname@#1}{\ClassWarning{\@classname}{Unknown edition (#1) of Mathematics Subject Classification; using '2010'.}}{\@xp\let\@xp\subjclassname\csname subjclassname@#1\endcsname}}

\renewcommand{\subjclassname}{\textup{2010} Mathematics Subject Classification}

\makeatother

\begin{document}

\title[On the lattice of overcommutative monoid varieties]{On the lattice of overcommutative\\
monoid varieties}

\thanks{The work is supported by the Russian Foundation for Basic Research (grant 17-01-00551) and by the Ministry of Education and Science of the Russian Federation (project 1.6018.2017/8.9).}

\author{S.\,V.\,Gusev}

\address{Ural Federal University, Institute of Natural Sciences and Mathematics, Lenina 51, 620000 Ekaterinburg, Russia}

\email{sergey.gusb@gmail.com}

\date{}

\begin{abstract}
It is unknown so far, whether the lattice of all varieties of monoids satisfies some non-trivial identity. The objective of this note is to give a negative answer to this question. Namely, we prove that any finite lattice is a homomorphic image of some sublattice of the lattice of overcommutative varieties of monoids (i.e., varieties that contain the variety of all commutative monoids). This implies that the lattice of overcommutative varieties of monoids, and therefore, the lattice of all varieties of monoids does not satisfy any non-trivial identity.
\end{abstract}

\keywords{monoid, variety, lattice of varieties}

\subjclass{Primary 20M07, secondary 08B15}

\maketitle

We study the lattice of varieties of monoids, i.e., algebras with two operations, namely an associative binary operation and a 0-ary operation that fixes the neutral element. There are many articles devoted to varieties of monoids. But the most part of them deals with the identities of the monoids. At the same time, only a little information is known so far about the lattice of monoid varieties which we denote by \textbf{MON}. In the works~\cite{Head-68} and~\cite{Wismath-86}, respectively, the lattices of all commutative and all idempotent monoid varieties are described. In~\cite[Theorem~1]{Pollak-81}, an example of a monoid variety without covers in the lattice \textbf{MON} is discovered. In~\cite[Subsection~3.2]{Jackson-Lee-17+} two monoid varieties are exhibited such that the subvariety lattices of both the varieties are finite, while the subvariety lattice of their join is uncountably infinite and does not satisfy the ascending chain condition. Also, there are a few papers where a description of subvariety lattices of some concrete monoid varieties appeared as auxiliary results (see~\cite[Lemma~4.1]{Lee-14}, for instance). These results probably form a majority of all currently known results regarding the lattice \textbf{MON}. This is in sharp contrast with a large number of striking and deep results about the lattice of semigroup varieties obtained so far (see the survey~\cite{Shevrin-Vernikov-Volkov-09}).

Due to insufficient knowledge of the lattice of monoid varieties, many natural questions here remain open so far. In particular, it is unknown so far, whether this lattice satisfies some non-trivial identity. The objective of this note is to give a negative answer to this question. Note that the negative answer to the analogous question for the lattice of semigroup varieties was found in 1971 in two articles by Burris and Nelson~\cite{Burris-Nelson-71a,Burris-Nelson-71b}.

A variety of monoids is called \emph{overcommutative} if it contains the variety of all commutative monoids. Evidently, the class of all overcommutative varieties forms a sublattice in the lattice of all monoid varieties, and we denote this sublattice by \textbf{OC}. The main result of this article is the following

\begin{theorem}
The lattice $\mathbf{OC}$ of all overcommutative monoid varieties does not satisfy any non-trivial identity.
\end{theorem}

We note that the analog of this theorem for semigroup varieties follows from~\cite[Corollary 4.4]{Volkov-94}.

Note also that the fact that the lattice \textbf{MON} does not satisfy any non-trivial identity was proved recently by I.Mikhailova. More precisely, she has proved that the lattice of all periodic monoid varieties contains an anti-isomorphic copy of the partition lattice over a countably infinite set (a private communication). It is well known that any non-trivial identity is false in the lattice with the above mentioned property~\cite{Sachs-61}. It may be mentioned here that the lattice \textbf{OC} does not contain an anti-isomorphic copy of the partition lattice over a countably infinite set. This follows from two observations. First, the analogous claim is true for the lattice of overcommutative semigroup varieties~\cite[Corollary~2.4]{Volkov-94}. Second, it is easy to see that the lattice \textbf{OC} is a sublattice of the lattice of overcommutative semigroup varieties. In this connection, the following open question seems to be natural.

\begin{question}
Does the lattice \textbf{OC} contain an anti-isomorphic copy of the partition lattice over any finite set?
\end{question}

To prove the main result, we need some definitions, notation and auxiliary assertions. As usual, we denote the principal ideal of a lattice $L$ generated by an element $a\in L$ by $(a]_L$. The partition lattice over the set $X$ is denoted by $\Part(X)$. 

The following lemma describes principal ideals of partition lattices. 

\begin{lemma}[{\cite[Lemma 403(v)]{Gratzer-11}}]
\label{ideals}
Let $\alpha$ be the partition of a set $X$ into two classes $A$ and $B$. Then the map $\beta$ from $(\alpha]_{\Part(X)}$ to $\Part(A)\times\Part(B)$ given by the rule
$$
\beta\longmapsto(\beta|_A,\beta|_B)\text{ for any }\beta\in(\alpha]_{\Part(X)}
$$
is a lattice isomorphism.\qed
\end{lemma}

First, we fix the notation. The free semigroup and the free monoid over the same countably infinite alphabet are denoted by $F$ and $F^1$ respectively. Two parts of identities we connect by the symbol $\approx$, while the symbol $=$ denotes the equality relation on $F^1$. Elements of the monoid $F^1$ are called \emph{words}. The empty word is denoted by the symbol $\lambda$. Words, unlike letters, are written in bold. As usual, $\End(F)$ and $\End(F^1)$ denote the monoid of endomorphisms of the semigroup $F$ and monoid $F^1$, respectively. The following statement is the specialization for monoids of a well-known universal-algebraic fact.

\begin{lemma}
\label{deduction}
The identity $\mathbf u\approx \mathbf v$ holds in the variety of monoids given by an identity system $\Sigma$ if and only if there exists a sequence of words $\mathbf u=\mathbf w_0,\mathbf w_1,\dots,\mathbf w_n=\mathbf v$ such that, for any $i\in\{0,1,\dots,n-1\}$, there are words $\mathbf a_i,\mathbf b_i\in F^1$, an endomorphism $\xi_i\in\End(F^1)$ and an identity $\mathbf u_i\approx \mathbf v_i$ from the system $\Sigma$ such that either $\mathbf w_i=\mathbf a_i\xi_i(\mathbf u_i)\mathbf b_i$ and $\mathbf w_{i+1}=\mathbf a_i\xi_i(\mathbf v_i)\mathbf b_i$ or $\mathbf w_i=\mathbf a_i\xi_i(\mathbf v_i)\mathbf b_i$ and $\mathbf w_{i+1}=\mathbf a_i\xi_i(\mathbf u_i)\mathbf b_i$.\qed
\end{lemma}

For a set $X$ of letters and a word $\bf w$, we denote by $\mathbf w_X$ the word that is obtained from $\bf w$ by deleting all letters from $X$. We denote by $\con(\mathbf w)$ the \emph{content} of the word \textbf w, i.e., the set of all letters occurring in $\mathbf w$. The number of occurrences of a letter $x$ in a word $\mathbf w$ is denoted by $\occ_x(\mathbf w)$. Further, we denote by $\mathcal L_{\FIC}(F^1)$ the lattice of all fully invariant congruences on $F^1$ and by $\FIC(\mathbf V)$ the fully invariant congruence on $F^1$ corresponding to the monoid variety $\mathbf V$. It is a general knowledge that the map $\FIC\colon\mathbf{MON}\longrightarrow\mathcal L_{\FIC}(F^1)$ is a lattice anti-isomorphism. For any $\mathbf u,\mathbf v\in F$, we define $\mathbf u\le \mathbf v$ if and only if $\mathbf v=\mathbf a\xi(\mathbf u)\mathbf b$ for some $\mathbf a,\mathbf b\in F^1$ and some $\xi\in\End(F)$. The relation $\le$ is a quasiorder on $F$. For an arbitrary anti-chain $A\subseteq F$, we consider the set $L_A$ of all monoid varieties $\mathbf V$ with the propery that $A$ is a union of $\FIC(\mathbf V)$-classes. We define the map $\varphi_A\colon L_A\longrightarrow\Part(A)$ by the rule $\varphi_A(\mathbf V)=\FIC(\mathbf V)|_A$ for each $\mathbf V\in L_A$.

The following lemma plays the key role in the proof of the main result. Note that the proof of this lemma is quite analogous to the proof of~\cite[Lemma 3]{Shaprynskii-12}.

\begin{lemma}
\label{anti-hom}
Let $A$ be an anti-chain in $F$ and, for arbitrary two words $\mathbf u,\mathbf v\in A$ and any non-empty set $X\subseteq \con(\mathbf u)$, the equalities $\con(\mathbf u)=\con(\mathbf v)$ and $\mathbf u_X = \mathbf v_X$ hold. Then:
\begin{itemize}
\item[\textup{(i)}] the set $L_A$ is a sublattice of the lattice $\mathbf{MON}$;
\item[\textup{(ii)}] the map $\varphi_A$ is a surjective anti-homomorpism of the lattice $L_A$ onto the lattice $\Part(A)$;
\item[\textup{(iii)}] for any partition $\beta\in\Part(A)$ there is an overcommutative monoid variety $\mathbf V\in L_A$ with $\varphi_A(\mathbf V)=\beta$.
\end{itemize}
\end{lemma}

\begin{proof}
(i) Let $\alpha$ be the partition of the monoid $F^1$ on two classes: $A$ and $F^1\setminus A$. Its principal ideal $(\alpha]_{\Part(F^1)}$ consists of all partitions $\beta$ such that $A$ is a union of $\beta$-classes. This observation and the definition of the set $L_A$ imply that $L_A$ is the pre-image of the sublattice $(\alpha]_{\Part(F^1)}\cap\mathcal L_{\FIC}(F^1)$ of the lattice $\mathcal L_{\FIC}(F^1)$ under the anti-homomorphism $\FIC$. Therefore, $L_A$ is a sublattice of the lattice $\mathbf{MON}$.

\smallskip

(ii) We define the map $\psi_A\colon(\alpha]_{\Part(F^1)}\longrightarrow\Part(A)$ by the rule $\psi_A(\beta)=\beta|_A$ for any $\beta\in(\alpha]_{\Part(F^1)}$. By Lemma \ref{ideals}, the lattice $(\alpha]_{\Part(F^1)}$ is decomposable into the direct product of the lattices $\Part(A)$ and $\Part(F^1\setminus A)$, and the map $\psi_A$ is the projection on the first factor. Hence $\psi_A$ is a lattice homomorphism. The map $\varphi_A$ is a composition of restriction of the anti-isomorphism $\FIC$ to $L_A$, and the restriction of the homomorphism $\psi_A$ to $(\alpha]_{\Part(F^1)}\cap\mathcal L_{\FIC}(F^1)$. Therefore, $\varphi_A$ is an anti-homomorphism. It remains to verify that $\varphi_A$ is surjective. Let $\beta\in\Part(A)$ and let $\mathbf V$ be the variety given by all identities of the form $\mathbf{u\approx v}$ with $(\mathbf u,\mathbf v)\in\beta$. We need to show check that $\mathbf V\in L_A$ and $\varphi_A(\mathbf V)=\beta$. This is true whenever every $\beta$-class is a $\FIC(\mathbf V)$-class. Thus, we need to verify that if an identity $\mathbf u\approx \mathbf v$ holds in $\mathbf V$ and $\mathbf u\in A$ then $(\mathbf u,\mathbf v)\in\beta$. By Lemma \ref{deduction} and induction, we can reduce our considerations to the case when $\mathbf u= \mathbf a\xi(\mathbf u')\mathbf b$ and $\mathbf v= \mathbf a\xi(\mathbf v')\mathbf b$ for some $\mathbf a,\mathbf b\in F^1$, $\xi\in\End(F^1)$ and a pair of words $(\mathbf u', \mathbf v')\in\beta$. Suppose that $\xi(x)=\lambda$ for some letter $x\in \con(\mathbf u')$. Then $\xi(\mathbf u')=\xi(\mathbf v')$ because $\mathbf u_X' = \mathbf v_X'$ for any non-empty set $X\subseteq \con(\mathbf u')$. Then $\mathbf u=\mathbf v$, whence $(\mathbf u,\mathbf v)\in\beta$. Suppose now that $\xi$ maps all letters from $\con(\mathbf u')$ into non-empty words. Then there is an endomorphism $\zeta\in \End(F)$ with $\xi|_{\con(\mathbf u')}=\zeta|_{\con(\mathbf u')}$. Then $\mathbf u= \mathbf a\zeta(\mathbf u')\mathbf b$. This means that $\mathbf u'\le \mathbf u$. Since $\mathbf u$ and $\mathbf u'$ lie in the anti-chain $A$, we have $\mathbf u=\mathbf u'$. Then $\mathbf a=\mathbf b=\lambda$ and $\xi(x)=x$ for any letter $x$ from the words $\mathbf u'$ and $\mathbf v'$. Therefore, $\mathbf v= \mathbf v'$. Thus, $\mathbf u=\mathbf u'\,\beta\,\mathbf v'=\mathbf v$, whence $(\mathbf u, \mathbf v)\in\beta$.

\smallskip

(iii) It follows from the proof of the claim (ii) that if $\beta\in\Part(A)$ and $\mathbf V$ is the monoid variety given by all identities of the form $\mathbf{u\approx v}$ with $(\mathbf u,\mathbf v)\in\beta$ then $\mathbf V\in L_A$ and $\varphi_A(\mathbf V)=\beta$. To prove the claim (iii), it remains to check that the variety $\bf V$ is overcommutative. Let $(\mathbf u,\mathbf v)\in \beta$ and $x\in\con(\mathbf u)$. Put $X=\con(\mathbf u)\setminus \{x\}$. By the hypothesis, $\mathbf u_X = \mathbf v_X$, whence $\occ_x(\mathbf u)=\occ_x(\mathbf v)$. Then the equality $\con(\mathbf u)=\con(\mathbf v)$ implies that $\occ_x(\mathbf u)=\occ_x(\mathbf v)$ for any letter $x$. It is well known that this implies the desired conclusion.
\end{proof}

Now we start with the direct proof of the main result. We denote by $\ell(\mathbf w)$ the length of the word \textbf w. Let $n$ be a natural number. We are going to check that the set of words  
$$
A_n = \{x^{n-i}yx^i \mid 0\le i \le n\}
$$ 
is an anti-chain. Suppose that this is not the case. Then there are different $i,j$ such that $x^{n-i}yx^i = \mathbf a\xi(x^{n-j}yx^j)\mathbf b$ for some $\mathbf a,\mathbf b \in F^1$ and some $\xi\in\End(F)$. Note that 
$$
\ell(x^{n-i}yx^i)=\ell(x^{n-j}yx^j)=n+1\le\ell(\xi(x^{n-j}yx^j)),
$$ 
whence $\mathbf a=\mathbf b =\lambda$. We see that the endomorphism $\xi$ maps letters to letters. It is easy to see that in this case $x^{n-i}yx^i\ne\xi(x^{n-j}yx^j)$. Thus, we have proved that the set of words $A_n$ is an anti-chain. Also, it is evident that $\mathbf u_{\{x\}}=\mathbf v_{\{x\}}=x^n$, $\mathbf u_{\{y\}}=\mathbf v_{\{y\}}=y$ and $\mathbf u_{\{x,y\}}=\mathbf v_{\{x,y\}}=\lambda$ for any $\mathbf u, \mathbf v \in A_n$. Now Lemma \ref{anti-hom}(i),(ii) implies that the map $\varphi_{A_n}$ is an anti-homomorphism of the lattice $L_{A_n}$ onto the lattice $\Part(A_n)$. Lemma \ref{anti-hom}(iii) implies that the restriction of the homomorphism $\varphi_{A_n}$ to the lattice $L_{A_n}\cap\bf OC$ is an anti-homomorphism of this lattice onto $\Part(A_n)$.  

We have proved that the partition lattice over an arbitrary finite set is an anti-homomorphic image of some sublattice of the lattice \textbf{OC}. Suppose that the last lattice satisfies some non-trivial identity $\varepsilon$. Then the identity dual to $\varepsilon$ holds in $\text{Part}(A_n)$ for any $n$. But it is well known that the class of all partition lattices of finite sets does not satisfy any non-trivial identity~\cite{Sachs-61}. This completes the proof of the main result.\qed

\subsection*{Acknowledgments.} The author is sincerely grateful to Professor Vernikov for his attention and assistance in the writing of the article, to Drs. Edmond W.H. Lee and Azeef Parayil for several valuable suggestions for improving the text, and to Dr. Shaprynski\v{\i} for helpful discussions.

\end{document}